\documentclass[10pt]{article}

\usepackage{amsmath}
\usepackage{amsfonts}
\usepackage{amssymb}
\usepackage{amscd}
\newcounter{theorem}
\renewcommand{\thetheorem}{\arabic{section}.\arabic{theorem}}

\newenvironment{thm}[1]{\par
\begin{sloppypar}\refstepcounter{theorem}%
\noindent{\bf #1 \thetheorem.}\it{}}{\end{sloppypar}}

\newenvironment{theorem}{\begin{thm}{Theorem}}{\end{thm}}
\newenvironment{proposition}{\begin{thm}{Proposition}}{\end{thm}}

\newenvironment{lemma}{\begin{thm}{Lemma}}{\end{thm}}

\newenvironment{defi}[1]{\par
\begin{sloppypar}\refstepcounter{theorem}%
\noindent{\bf #1 \thetheorem.}\rm{}}{\end{sloppypar}}
\newenvironment{definition}{\begin{defi}{Definition}}{\end{defi}}

\newenvironment{remark}{\begin{defi}{Remark}}{\end{defi}}
\newenvironment{hypothesis}{\begin{defi}{Hypothesis}}{\end{defi}}

\newcommand{\eh}{\hfill}\newlength{\sperr}
\newenvironment{proof}{{\settowidth{\sperr}{\rm Proof}
\par\addvspace{0.3cm}\noindent\parbox[t]{1.3\sperr}{\rm P\eh r\eh o\eh o\eh
f\eh.}}}{\nopagebreak\mbox{}\hfill $\blacksquare $\par\addvspace{0.25cm}}

\newenvironment{proof_}{{\settowidth{\sperr}{\rm Proof}
\par\addvspace{0.3cm}\noindent\parbox[t]{1.3\sperr}{\rm P\eh r\eh o\eh o\eh
f\eh}}}{\nopagebreak\mbox{}\hfill $\blacksquare $\par\addvspace{0.25cm}}

\def\R{{\rm I\kern-.2em R}}

\def\dbar{\;\;\bar{}\!\!\!d}

\def\Ker{\text{\rm Ker}}

\begin{document}

\title{Eigenfunctions decay for magnetic pseudodifferential operators}

\date{\today}

\author{Viorel Iftimie\footnote{Institute
of Mathematics Simion Stoilow of the Romanian Academy, P.O.  Box
1-764, Bucharest, Romania.} \ \ and Radu
Purice\footnote{Institute
of Mathematics Simion Stoilow of the Romanian Academy, P.O.  Box
1-764, Bucharest, Romania.}
\footnote{Laboratoire Europ\'een Associ\'e CNRS Franco-Roumain {\it Math-Mode}}}

\maketitle

\begin{abstract}

We prove rapid decay (even exponential decay under some stronger assumptions) of the eigenfunctions associated to discrete eigenvalues, for a class of self-adjoint operators in $L^2(\mathbb{R}^d)$ defined by ``magnetic'' pseudodifferential operators (studied in \cite{IMP1}). This class contains the relativistic Schr\"{o}dinger operator with magnetic field.

\end{abstract}

\section{Introduction}

Making a complete list of the papers devoted to the problem of exponential decay of the eigenfunctions of elliptic differential operators being quite impossible and shurely out of the aim of this paper, let us just mention the basic paper by Agmon \cite{Ag} and the one by Rabier \cite{Ra} that has influenced our paper and that also contain many other references. For Schr\"{o}dinger operators with magnetic fields such results have been obtained in \cite{HN,I,Sh}, etc. There exist also some results concerning the exponential decay of eigenfunctions for some  pseudodifferential operators (without magnetic fields), as for example the relativistic Schr\"{o}dinger operator \cite{N,CMS,HP} or the Kac operator \cite{H}.

In this paper we prove rapid decay (and under some more assumptions even exponential decay) for the eigenfunctions associated to isolated finite multiplicity eigenvalues of the self-adjoint realizations in $L^2(\mathbb{R}^d)$ of the ``magnetic'' pseudodifferential operators introduced in a series of papers \cite{KO1,KO2,Mu,MP1} and thoroughly studied in \cite{IMP1}. In order to formulate our results and their proofs we shall need to quikly recall some facts from this last paper.

Let us begin by formulating our assumptions concerning the magnetic field $B$. We denote by $BC^\infty(\mathbb{R}^n)$ the space of infinitely differentiable functions bounded together with all their derivatives
$$
BC^\infty(\mathbb{R}^n):=\left\{f\in C^\infty(\mathbb{R}^n)\mid\partial^\alpha f\in L^\infty(\mathbb{R}^n)\forall\alpha\in\mathbb{N}^n\right\}.
$$
We shall always assume that
\begin{equation}\label{Hyp-B}
 B=\frac{1}{2}\sum_{1\leq j,k\leq d}B_{jk}\,dx_j\wedge dx_k,\quad dB=0,\quad B_{jk}=-B_{kj}\in BC^\infty(\mathbb{R}^d).
\end{equation}

Let us recall that using the ``transversal'' gauge one can easily define a vector potential $A$ such that
\begin{equation}\label{Def-A}
 A=\sum_{1\leq j\leq d}A_j\,dx_j,\quad A_j\in C^\infty_{\text{\sf pol}}(\mathbb{R}^d),\quad B=dA
\end{equation}
with $C^\infty_{\text{\sf pol}}(\mathbb{R}^d)$ the space of infinitely differentiable functions with at most polynomial growth together with all their derivatives. We shall use the notation
\begin{equation}\label{Def-omega}
 \omega^A(x,y):=\exp\left(-i\int_{[x,y]}A\right),\qquad\forall(x,y)\in\mathbb{R}^d\times\mathbb{R}^d.
\end{equation} 

The above cited papers propose a gauge covariant formalism for associating a quantum observable $\mathfrak{Op}^A(a)$ to any classical observable $a:\mathbb{R}^d\times\mathbb{R}^d\rightarrow\mathbb{R}$. For $B=0$ this formalism reduces to the usual Weyl quantization procedure. We shall call the operators $\mathfrak{Op}^A(a)$ {\it magnetic pseudodifferential operators}. Usually we shall suppose that $a$ is a H\"{o}rmander type symbol. Let us recall their definition.
\begin{definition}\label{Def-H-symb}
 For any $m\in\mathbb{R}$ we consider the complex linear space 
$$
S^m(\mathbb{R}^N\times\mathbb{R}^n):=\left\{a\in C^\infty(\mathbb{R}^N\times\mathbb{R}^n)\mid\underset{(x,\xi)\in\mathbb{R}^N\times\mathbb{R}^n}{\sup}<\xi>^{-m+|\beta|}\left|\big(\partial^\alpha_x\partial^\beta_\xi a\big)(x,\xi)\right|<\infty,\ \forall\alpha\in\mathbb{N}^N,\,\forall\beta\in\mathbb{N}^n\right\}
$$
where $<\xi>:=\sqrt{1+|\xi|^2}$ for any $\xi\in\mathbb{R}^n$. This linear space is endowed with its natural Fr\'{e}chet topology given by the seminorms appearing in the above definition. We shall use the shorthand notation $S^m(\mathbb{R}^d)\equiv S^m(\mathbb{R}^d\times\mathbb{R}^d)$.
\end{definition}

For any $a\in S^m(\mathbb{R}^d)$ one has that $\mathfrak{Op}^A(a)$ belongs to $\mathbb{B}\big(\mathcal{S}(\mathbb{R}^d)\big)$ and by duality also to $\mathbb{B}\big(\mathcal{S}^\prime(\mathbb{R}^d)\big)$ and is explicitely given by the following oscillating integral
\begin{equation}\label{Def-OpA}
 \left[\mathfrak{Op}^A(a)u\right](x):=\int_{\mathbb{R}^{2d}}e^{i<x-y,\eta>}\omega^A(x,y)a\left(\frac{x+y}{2},\eta\right)u(y)\,dy\,\dbar\eta,\quad\forall x\in\mathbb{R}^d,
\end{equation}
where $u\in\mathcal{S}(\mathbb{R}^d)$ and $\dbar\eta:=(2\pi)^{-d}d\eta$.

In particular, if $a$ is a polynomial in the second variable $\xi$ then $\mathfrak{Op}^A(a)$ is a differential operator; for example if $a(x,\xi)=|\xi|^2$, then $\mathfrak{Op}^A(a)$ is the Schr\"{o}dinger operator with magnetic field $B$.

If $a\in S^m(\mathbb{R}^d)$ and $b\in S^{m^\prime}(\mathbb{R}^{d})$ then there exists a unique symbol $a\sharp^Bb\in S^{m+m^\prime}(\mathbb{R}^d)$ such that $\mathfrak{Op}^A(a\sharp^Bb)=\mathfrak{Op}^A(a)\mathfrak{Op}^A(b)$ and the map
$$
S^m(\mathbb{R}^d)\times S^{m^\prime}(\mathbb{R}^{d})\ni(a,b)\mapsto a\sharp^Bb\in S^{m+m^\prime}(\mathbb{R}^d)
$$
is continuous.

For any $s\in\mathbb{R}$ we denote by $p_s(\eta):=<\eta>^s,\,\forall\eta\in\mathbb{R}^d$. Then $p_s\in S^s(\mathbb{R}^d)$ and we define $P_s:=\mathfrak{Op}^A(p_s)$. For $s\geq0$ we define
\begin{equation}\label{Def-sp-Sob}
 \mathcal{H}^s_A(\mathbb{R}^d):=\left\{u\in L^2(\mathbb{R}^d)\mid P_su\in L^2(\mathbb{R}^d\right\},
\end{equation} 
that is a complex Hilbert space for the natural scalar product:
\begin{equation}\label{Def-pr-sc-Sob}
 \left<u,v\right>_{\mathcal{H}^s_A}:=\left<u,v\right>_{L^2}+\left<P_su,P_sv\right>_{L^2}.
\end{equation}
For $s<0$ we denote by $\mathcal{H}^s_A(\mathbb{R}^d)$ the dual of $\mathcal{H}^{-s}_A(\mathbb{R}^d)$ and by $\mathcal{H}^{\infty}_A(\mathbb{R}^d)$ the intersection $\underset{s\in\mathbb{R}}{\bigcap}\mathcal{H}^s_A(\mathbb{R}^d)$. In \cite{IMP1} these spaces are called {\it magnetic Sobolev spaces}.

The following statements are true:
\begin{itemize}
 \item The space $\mathcal{S}(\mathbb{R}^d)$ is densely embeded in $\mathcal{H}^s_A(\mathbb{R}^d)$ for any $s\in\mathbb{R}$.
\item For $s_1\leq s_2$ we have that the space $\mathcal{H}^{s_2}_A(\mathbb{R}^d)$ is continuously embedded in $\mathcal{H}^{s_1}_A(\mathbb{R}^d)$.
\item For $s\in\mathbb{N}$ we have the equality
\begin{equation}\label{sp-Sob-intregi}
 \mathcal{H}^s_A(\mathbb{R}^d)=\left\{u\in L^2(\mathbb{R}^d)\mid (D-A)^\alpha u\in L^2(\mathbb{R}^d)\ \forall\alpha\in\mathbb{N}^d,\,|\alpha|\leq s\right\}
\end{equation}
where $(D-A)^\alpha=(D_1-A_1)^{\alpha_1}\cdots(D_d-A_d)^{\alpha_d}$ for $\alpha=(\alpha_1,\ldots,\alpha_d)$.
\item For $a\in S^m(\mathbb{R}^d)$ we have that $\mathfrak{Op}^A(a)$ is a bounded operator from $\mathcal{H}^s_A(\mathbb{R}^d)$ to $\mathcal{H}^{s-m}_A(\mathbb{R}^d)$ for any $s\in\mathbb{R}$ and the map
$$
S^m(\mathbb{R}^d)\ni a\mapsto\mathfrak{Op}^A(a)\in\mathbb{B}\big(\mathcal{H}^s_A(\mathbb{R}^d);\mathcal{H}^{s-m}_A(\mathbb{R}^d)\big)
$$
is continuous.
\end{itemize}

\begin{hypothesis}\label{Hyp-symb}
 We shall always assume that  $a\in S^m(\mathbb{R}^d)$ is a real valued function and for $m>0$ we shall suppose it to be elliptic also, i.e. we suppose that there exist two positive constants $C$ and $R$ such that 
$$
\left|a(x,\xi)\right|\geq C<\xi>^m\ \forall(x,\xi)\in\mathbb{R}^d\times\mathbb{R}^d \text{ with }|\xi|\geq R.
$$
\end{hypothesis}
\begin{remark}\label{Rem-symb}
 Under the above Hypothesis it is proved in \cite{IMP1} that $\mathfrak{Op}^A(a)$ considered as linear operator in $L^2(\mathbb{R}^d)$ with domain $\mathcal{S}(\mathbb{R}^d)$ is essentialy self-adjoint. Its closure that we shall denote by $H$ is a self-adjoint operator with domain $D(H)=L^2(\mathbb{R}^d)$ for $m\leq0$ and respectively $D(H)=\mathcal{H}^m_A(\mathbb{R}^d)$ for $m>0$. In this last case the topology of $\mathcal{H}^m_A(\mathbb{R}^d)$ coincides with the graph-norm of $H$.
\end{remark}

The first Theorem of this paper states that the eigenfunctions associated to the discrete spectrum of $H$ have rapid decay.
\begin{theorem}\label{Main-1}
 Let us suppose that Hypothesis \ref{Hyp-symb} is verified and that the magnetic field verifies (\ref{Hyp-B}) and let us denote by $H$ the self-adjoint operator in $L^2(\mathbb{R}^d)$ associated to $\mathfrak{Op}^A(a)$ as above (Remark \ref{Rem-symb}). Let $\lambda\in\sigma_{\text{disc}}(H)$ and $u\in\Ker(H-\lambda)$. Then
\begin{enumerate}
 \item $<x>^pu\in\underset{n\in\mathbb{N}}{\bigcap}D(H^n)\ \forall p\in\mathbb{N}$.
\item If $m>0$ or if $m<0$ and $\lambda\ne0$ then $u\in\mathcal{S}(\mathbb{R}^d)$.
\end{enumerate}
\end{theorem}

In order to obtain exponential decay for the eigenfunctionswe shall need to add a hypothesis implying the existence of an analytic extension of the function $a(x,\cdot)$, for any $x\in\mathbb{R}^d$, to the domain $D_\delta\in\mathbb{C}^d$, where for $\delta>0$ we denote by
\begin{equation}\label{Def-olo-band}
 D_\delta:=\left\{\zeta=(\zeta_1,\ldots,\zeta_d)\in\mathbb{C}^d\mid |\Im\zeta_j|<\delta,\,1\leq j\leq d\right\}.
\end{equation}
\begin{hypothesis}\label{Hyp-symb-olo}
Let $a\in S^m(\mathbb{R}^d)$ and suppose that there exists $\delta>0$ and a function $\widetilde{a}:\mathbb{R}^d\times D_\delta\rightarrow\mathbb{C}$ such that:
\begin{enumerate}
 \item for any $x\in\mathbb{R}^d$ the function $\widetilde{a}(x,\cdot):D_\delta\rightarrow\mathbb{C}$ is analytic;
\item the map $\mathbb{R}^d\times\mathbb{R}^d\ni(x,\eta)\mapsto\widetilde{a}(x,\eta+i\xi)\in\mathbb{C}$ is of class $S^m(\mathbb{R}^d)$ uniformly (for the Fr\'{e}chet topology) with respect to $\xi=(\xi_1,\ldots,\xi_d)\in\mathbb{R}^d$ for $|\xi_j|<\delta$, $1\leq j\leq d$;
\item we have: $a=\left.\widetilde{a}\right|_{\mathbb{R}^d\times\mathbb{R}^d}$.
\end{enumerate}
\end{hypothesis}

\begin{remark}
 Using the Cauchy formula for a polydisc, one can easily prove that if $a(x,\cdot)$ has an analytic extension to a ``conic'' neighbourhood $\Gamma_\delta$ of $\mathbb{R}^d\subset\mathbb{C}^d$ defined as
$$
\Gamma_\delta:=\left\{\zeta=(\zeta_1,\ldots,\zeta_d)\in\mathbb{C}^d\mid |\Im\zeta_j|<\delta<\Re \zeta>,\,1\leq j\leq d\right\},
$$
then Hypothesis \ref{Hyp-symb-olo} is a consequence of the following simpler hypothesis:
\end{remark}
\begin{hypothesis}\label{Hyp-symb-olo-2}
Let $a\in S^m(\mathbb{R}^d)$ and suppose that there exist $\delta>0$ and a function $\widetilde{a}\in C^\infty(\mathbb{R}^d\times \Gamma_\delta)$ such that:
\begin{enumerate}
 \item for any $x\in\mathbb{R}^d$ the function $\widetilde{a}(x,\cdot):D_\delta\rightarrow\mathbb{C}$ is analytic on $\Gamma_\delta$;
\item for any $\alpha\in\mathbb{N}^d$ there exists a constant $C_\alpha>0$ such that $\left|\big(\partial^\alpha_x\widetilde{a}\big)(x,\xi)\right|\leq\,C_\alpha<\Re\zeta>^m$ on $\mathbb{R}^d\times \Gamma_\delta$;
\item we have: $a=\left.\widetilde{a}\right|_{\mathbb{R}^d\times\mathbb{R}^d}$.
\end{enumerate}
\end{hypothesis}
Evidently the symbols $a$ that are polynomials in the second variable with coefficients of class $BC^\infty(\mathbb{R}^d)$ and the symbols $p_s$ (for $s\in\mathbb{R}$) satisfy Hypothesis \ref{Hyp-symb-olo-2}.

\begin{theorem}\label{Main-2}
 Let us suppose that Hypothesis \ref{Hyp-symb}and \ref{Hyp-symb-olo} are verified and that the magnetic field verifies (\ref{Hyp-B}). Let us denote by $H$ the self-adjoint operator in $L^2(\mathbb{R}^d)$ associated to $\mathfrak{Op}^A(a)$ as above (Remark \ref{Rem-symb}). Let $\lambda\in\sigma_{\text{disc}}(H)$ and $u\in\Ker(H-\lambda)$. Then there exists $\epsilon_0>0$ such that for any $\epsilon\in(0,\epsilon_0]$ we have that
\begin{enumerate}
 \item $e^{\epsilon<x>}u\in\underset{n\in\mathbb{N}}{\bigcap}D(H^n)$.
\item If $m>0$ or if $m<0$ and $\lambda\ne0$ then $e^{\epsilon<x>}u\in\mathcal{S}(\mathbb{R}^d)$.
\end{enumerate}
\end{theorem}

Results similar to the above two Theorems can also be obtained for singular perturbations of some magnetic pseudodifferential operators. Let us ilustrate this procedure on the case of operators of the form $P_1+V$ with $P_1=\mathfrak{Op}^A(p_1)$ with $p_1(\eta)=<\eta>$. Let us notice that for $B=0$ this is just the relativistic Schr\"{o}dinger operator and has been studied in \cite{CMS,HP}.

As we have already noticed the symbol $p_1$ verifies the Hypothesis \ref{Hyp-symb} and \ref{Hyp-symb-olo}. As in our Remark \ref{Rem-symb} we shall denote by $H_A$ the self-adjoint operator in $L^2(\mathbb{R}^d)$ associated to $P_1$, having the domain $\mathcal{H}^1_A(\mathbb{R}^d)$. Concerning $V$ we shall use the following Hypothesis.
\begin{hypothesis}\label{Hyp-V}
 We suppose that $V:\mathbb{R}^d\rightarrow\mathbb{R}$ has the decomposition $V=V_+-V_-$ with $V_\pm\geq0$, $V_\pm\in L^1_{\text{\sf loc}}(\mathbb{R}^d)$ and that multiplication with $V_-$ on $L^2(\mathbb{R}^d)$ is a form relatively bounded operator with respect to $H_0=\sqrt{1-\Delta}$ with relative bound strictly less then 1.
\end{hypothesis}

Now let us recall that as we have proved in \cite{IMP2}, under the Hypothesis \ref{Hyp-V} and assuming (\ref{Hyp-B}), one can define the ``form sum'' $H:=H_A\dot{+}V$. The domain of the form $h$ associated to $H$ is then 
$$
D(h)=\left\{u\in\mathcal{H}^{1/2}_A(\mathbb{R}^d)\mid V_+u\in L^2(\mathbb{R}^d)\right\}.
$$
An element $u\in D(h)$ is then also in $D(H)$ if and only if $P_1u+Vu\in L^2(\mathbb{R}^d)$, and in this case we have that $Hu=P_1u+Vu$.

Under the above Hypoyhesis \ref{Hyp-V} and assuming (\ref{Hyp-B}) one can prove $L^2$ exponential decay for the eigenfunctions of $H$ associated to the discrete spectrum. In order to obtain pointwise exponential decay one has to suppose that $V_-$ is of Kato class $\mathcal{K}_d$ associated to the operator $H_0$. Let us briefly recall its definition; following \cite{IT} and \cite{CMS} 
we consider the semigoup generated by $H_0$ that is explicitely given as convolution with the function
$$
\mathfrak{p}_t(x):=(2\pi)^{-\frac{d+1}{2}}2te^t\left(|x|^2+t^2\right)^{-\frac{d+1}{4}}K_{\frac{d+1}{2}}\left(\sqrt{|x|^2+t^2}\right),\quad\forall x\in\mathbb{R}^d,\,t>0,
$$
with $K_{\frac{d+1}{2}}$ the modified Bessel function of 3-rd type and order $\frac{d+1}{2}$. {\it Then a positive  function $W\in L^1_{\text{\sf loc}}(\mathbb{R}^d)$ belongs to $\mathcal{K}_d$ when it verifies the following equality:}
$$
\underset{t\searrow0}{\lim}\ \underset{x\in\mathbb{R}^d}{\sup}\int_0^t\left(\int_{\mathbb{R}^d}\mathfrak{p}_s(x-y)W(y)dy\right)ds\,=\,0.
$$
As proved in \cite{vCa,CMS,DvC}, if $W\in\mathcal{K}_d$, then multiplication by $W$ in $L^2(\mathbb{R}^d)$ is form relatively bounded with respect to $H_0$ with relative bound 0.
\begin{theorem}\label{Main-3}
 Under the Hypothesis \ref{Hyp-V} and assuming (\ref{Hyp-B}) let $H=H_A\dot{+}V$ (as above), let $\lambda\in\sigma_{\text{disc}}(H)$ and $u\in\Ker(H-\lambda)$. Then there exists a constant $\epsilon_0\in(0,1)$ such that for any $\epsilon\in(0,\epsilon_0]$ one has that:
\begin{enumerate}
 \item $e^{<\epsilon x>}u\in\underset{n\in\mathbb{N}}{\bigcap}D(H^n)$.
\item If $V_-\in\mathcal{K}_d$ then $e^{<\epsilon x>}u\in\ L^\infty(\mathbb{R}^d)$.
\end{enumerate}
\end{theorem}

The proofs of the above three theorems is based on the following fact, that is obvious in the case of differential operators:

{\it Either $f_\epsilon(x):=<\epsilon x>^p$ or $f_\epsilon(x):=e^{<\epsilon x>}$, for any $p\in\mathbb{N}$ and sufficiently small $\epsilon>0$, verifies: $f_\epsilon Hf_\epsilon^{-1}=H+\epsilon R_\epsilon$ with $R_\epsilon$ a $H$-relatively bounded operator uniformly for $\epsilon\searrow0$.}

This idea, that is also present in \cite{H} and \cite{Ra}, has to be combined with a perturbative argument inspired from \cite{Ra}. The proof of the point 2 of Theorem \ref{Main-3} needs also a 'diamagnetic' type inequality that we obtained in \cite{IMP2}.

The paper is structured as follows: in the first Section we obtain an abstract result about the behaviour of eigenfunctions under certain conditions and the following three sections deal each one with one of the above three Theorems.

\section{An abstract weighted estimation}

\begin{hypothesis}\label{Hyp-weight}
Let $H$ be a self-adjoint operator in $L^2(\mathbb{R}^d)$. Let $(R_\epsilon)_{\epsilon\in(0,1]}$ be a family of operators acting in $L^2(\mathbb{R}^d)$ and let $f:\mathbb{R}^d\rightarrow[1,\infty)$ be a measurable function. Let $f_\epsilon(x):=f(\epsilon x)$ for any $\epsilon>0$. We suppose that:
\begin{enumerate}
 \item the operators $R_\epsilon$ are $H$-relatively bounded uniformly in $\epsilon\in(0,1]$;
\item for any $u\in D(H)$ and $\epsilon\in(0,1]$ we have that $f_\epsilon^{-1} u\in D(H)$;
\item on $D(H)$ we have the equality $H_\epsilon:=f_\epsilon Hf_\epsilon^{-1}=H+\epsilon R_\epsilon$ for any $\epsilon\in(0,1]$.
\end{enumerate}
\end{hypothesis}

\begin{proposition}\label{P-1_1}
 Suppose that our Hypothesis \ref{Hyp-weight} is verified and suppose $\lambda\in\sigma_{\text{disc}}(H)$. Then there exists $\epsilon_0\in(0,1]$ such that for any $\epsilon\in(0,\epsilon_0]$ we have that $f_\epsilon u\in D(H^n)$ for any $n\geq1$ and any $u\in\Ker(H-\lambda)$.
\end{proposition}
The proof of this statement uses some ideas from \cite{Ra} and relies on the following lemma.

\begin{lemma}\label{L-1_2}
 Under the assumptions of Proposition \ref{P-1_1} there exist $\epsilon_0\in(0,1]$ and an open disc $D\subset\mathbb{C}$ with center $\lambda$ such that for any $\epsilon\in(0,\epsilon_0]$ the operator $H_\epsilon$ (defined on $D(H)$) is closed and 
$$
\overline{D}\cap\left(\sigma(H)\cup\sigma(H_\epsilon)\right)\,=\,\{\lambda\}.
$$
Moreover $\lambda$ is an eigenvalue of $H_\epsilon$ having an algebraic multiplicity equal to its multiplicity as eigenvalue of $H$.
\end{lemma}
\begin{proof}
 Due to point 1 of Hypothesis \ref{Hyp-weight}, if $\epsilon_0\in(0,1]$ is small enough and $\epsilon\in(0,\epsilon_0]$ then the operator $\epsilon R_\epsilon$ is $H$-relatively bounded with relative bound strictly less then 1. Thus, $H$ being a closed operator (self-adjoint in fact) and $H_\epsilon=H+\epsilon R_\epsilon$, the Theorem IV.1.1 in \cite{K} implies that $H_\epsilon$ is also a closed operator.

Let $m$ be the multiplicity of $\lambda$ as an eigenvalue of $H$. Using \S V.4.3 in \cite{K} we conclude that for $\epsilon_0$ small enough one can find an open disc $D\subset\mathbb{C}$ with center $\lambda$ such that $\overline{D}\bigcap\sigma(H)=\{\lambda\}$ and such that we also have 
$$
\overline{D}\cap\sigma(H_\epsilon)=D\cap\sigma(H_\epsilon)=\left\{\lambda_{\epsilon,1},\ldots,\lambda_{\epsilon,q(\epsilon)}\right\},\quad\forall\epsilon\in(0,\epsilon_0]
$$
where for any $j\in\{1,\ldots,q(\epsilon)\}$, $\lambda_{\epsilon,j}$ is an isolated eigenvalue of $H_\epsilon$ with algebraic multiplicity equal to $m_{\epsilon,j}$ and we have
$$
m_{\epsilon,1}+\ldots+m_{\epsilon,q(\epsilon)}\,=\,m.
$$
We shall prove now that we have in fact: $q(\epsilon)=1$ and $\lambda_{\epsilon,1}=\lambda$. Suppose that this would not be true; then $H_\epsilon$ would have an eigenvalue $\lambda_\epsilon\in D\setminus\{\lambda\}$. Let $v_\epsilon\in D(H)$ be an eigenfunction of $H_\epsilon$ associated to $\lambda_\epsilon$. Due to points 2 and 3 of our Hypothesis \ref{Hyp-weight} we conclude that 
$$
u_\epsilon:=f_\epsilon^{-1}v_\epsilon\in D(H), \text{ and }
(H-\lambda_\epsilon)u_\epsilon=f^{-1}_\epsilon(H_\epsilon-\lambda_\epsilon)v_\epsilon=0.
$$
This means that $\lambda_\epsilon$ should also be an eigenvalue of $H$ and this is false by hypothesis.
\end{proof}

\begin{proof_} {\sf (of Proposition \ref{P-1_1})}

 Let us consider $\epsilon_0\in(0,1]$ and the disc $D$ from Lemma \ref{L-1_2}. We denote by $H_0:=H$ and for any $\epsilon\in[0,\epsilon_0]$
$$
P_\epsilon:=(2\pi i)^{-1}\int_{\partial D}(H_\epsilon-\mu)^{-1}d\mu,
$$
the orthogonal projection in $L^2(\mathbb{R}^d)$ having an image denoted by $\mathcal{M}_\epsilon$. Using \S III.6 from \cite{K} we conclude that $\dim\mathcal{M}_\epsilon=m$ (the multiplicity of the eigenvalue $\lambda$ of $H$) and there exists $n_\epsilon\in\mathbb{N}^*$ such that $\mathcal{M}_\epsilon=\Ker(H_\epsilon-\lambda)^n$ for any $n\geq n_\epsilon$. In particular $n_0=1$.

Let us notice that if $v\in\mathcal{M}_\epsilon$ then we also have that $v\in D(H^n)$ for any $n\geq1$ and $(H_\epsilon-\lambda)^{n_\epsilon}v=0$. We deduce that $u:=f_\epsilon^{-1}v\in D(H^n)$ for any $n\geq1$ and 
$$
(H-\lambda)^{n_\epsilon}u=f_\epsilon^{-1}(H_\epsilon-\lambda)^{n_\epsilon}v=0
$$
concluding that $u\in\Ker(H-\lambda)$. Thus, taking any basis $\{v_1,\ldots,v_m\}$ for $\mathcal{M}_\epsilon$ and denoting by $u_j:=f_\epsilon^{-1}v_j$ for $j\in\{1,\ldots,m\}$, we obtain a basis $\{u_1,\ldots,u_m\}$ for $\Ker(H-\lambda)$. We conclude that any $u\in\Ker(H-\lambda)$ is of the form $u=f_\epsilon^{-1}v$ with $v\in\mathcal{M}_\epsilon\subset\underset{n\geq1}{\bigcap}D(H^n)$.
\end{proof_}

\section{Proof of Theorem \ref{Main-1}}

The proof of our Theorems relies on the magnetic pseudodifferential calculus we have developped in \cite{MP1,IMP1,IMP3}. The behaviour with respect to the small parameter $\epsilon\in(0,\epsilon_0]$ is achieved by using some asymptotic expansions near $\epsilon=0$ and the control of the remainders puts into evidence in a natural way a class of symbols with thre variables. In fact it seems easier to sistematically work with a pseudodifferential calculus associated to such symbols. 
\begin{definition}\label{psd-ampl}
 For $a\in S^m(\mathbb{R}^{2d}\times\mathbb{R}^d)$ we consider the linear operator defined by the following oscillatory integral:
$$
\left[\mathfrak{E}(a)u\right](x):=\int_{\mathbb{R}^{2d}}e^{i<x-y,\eta>}\omega^A(x,y)a(x,y,\eta)u(y)\,dy\,\dbar\eta,\qquad\forall u\in\mathcal{S}(\mathbb{R}^d),\quad\forall x\in\mathbb{R}^d.
$$
\end{definition}
One can use exactly the arguments in \cite{IMP1} in order to prove that for $a\in S^m(\mathbb{R}^{2d}\times\mathbb{R}^d)$ the operator 
$\mathfrak{E}(a)$ belongs in fact to $\mathbb{B}\big(\mathcal{S}(\mathbb{R}^d)\big)$, but the following result shows that in fact these new operators are in fact magnetic pseudodifferential operators in the sense of \cite{MP1, IMP1} for some associate symbol in $S^m(\mathbb{R}^d)$.

\begin{lemma}\label{L-2_1}
 Let $a\in S^m(\mathbb{R}^{2d}\times\mathbb{R}^d)$ for some $m\in\mathbb{R}$ and let $t\in[0,1]$. Then there exists a unique symbol $\overset{\circ}{a}_t\in S^m(\mathbb{R}^d)$ such that $\mathfrak{E}(a)=\mathfrak{E}(\overset{\circ}{a}_t\circ L_t)$ with $L_t:\mathbb{R}^{2d}\times\mathbb{R}^d\rightarrow\mathbb{R}^d\times\mathbb{R}^d$ given by $L_t(x,y,\eta):=(tx+(1-t)y,\eta)$ and the map
$$
S^m(\mathbb{R}^{2d}\times\mathbb{R}^d)\ni a\mapsto \overset{\circ}{a}_t\in S^m(\mathbb{R}^d)
$$is continuous.
\end{lemma}
\begin{remark}\label{Rem-L-2_1}
 Taking $t=1/2$ in the above Lemma we obtain that for any $a\in S^m(\mathbb{R}^{2d}\times\mathbb{R}^d)$ there exists a unique $\overset{\circ}{a}\in S^m(\mathbb{R}^{d})$ such that $\mathfrak{E}(a)=\mathfrak{Op}^A(\overset{\circ}{a})$.
\end{remark}
\begin{proof}
 The distribution kernel of the linear operator $\mathfrak{E}(a)$ is given explicitely by the following oscillatory integral (as element in $\mathcal{S}^\prime(\mathbb{R}^{2d})$:
$$
K_a(x,y):=\omega^A(x,y)\int_{\mathbb{R}^d}e^{i<x-y,\eta>}a(x,y,\eta)\,\dbar\eta.
$$
More precisely, choosing $\chi\in C^\infty_0(\mathbb{R}^d)$ with $\chi(0)=1$ and denoting by $\chi_\epsilon(\eta):=\chi(\epsilon\eta)$ for any $\epsilon\in[0,1]$, we define
$$
K_a(x,y):=\underset{\epsilon\searrow0}{\lim}K_{a_\epsilon}\ \text{in }\mathcal{S}^\prime(\mathbb{R}^{2d}),\quad\text{for }a_\epsilon(x,y,\eta):=\chi_\epsilon(\eta)a(x,y,\eta).
$$
Fixing now $t\in[0,1]$ and some $b\in S^m(\mathbb{R}^d)$ we notice that we have the following equality in $\mathcal{S}^\prime(\mathbb{R}^{2d})$
$$
K_{b\circ L_t}=\omega^AS_t^{-1}\mathcal{F}_2^{-1}b
$$
where $\mathcal{F}_2$ is the partial Fourier transform with respect to the second variable and $S_t$ is the linear change of variables $S_tf(x,y):=f(x+(1-t)y,x-ty)$. This equality allows us to extend the map 
$$
S^m(\mathbb{R}^d)\ni b\mapsto K_{b\circ L_t}\in\mathcal{S}^\prime(\mathbb{R}^{2d})
$$
to a linear topological isomorphism
$$
\mathcal{S}^\prime(\mathbb{R}^{2d})\ni b\overset{\mathfrak{L}_t}{\mapsto} K_{b\circ L_t}\in\mathcal{S}^\prime(\mathbb{R}^{2d}).
$$

For any $a\in S^m(\mathbb{R}^{2d}\times\mathbb{R}^d)$ let us define $\overset{\circ}{a}_t:=\mathfrak{L}_t^{-1}\big(K_a\big)$. Proving the Lemma now reduces to prove that this $\overset{\circ}{a}_t$ in fact belongs to $S^m(\mathbb{R}^d)$. For that we shall use all the above conclusions in order to write
$$
\overset{\circ}{a}_t=\underset{\epsilon\searrow0}{\lim}\mathfrak{L}_t^{-1}\big(K_{a_\epsilon}\big)
$$
and transform the oscillating integrals in some usual convergent integrals by using the equalities:
$$
e^{i<z,\zeta>}=<z>^{-2N_1}(1-\Delta_\zeta)^{N_1}e^{i<z,\zeta>}=<\zeta>^{-2N_2}(1-\Delta_z)^{N_2}e^{i<z,\zeta>},\quad\forall(z,\zeta)\in\mathbb{R}^{2d},\,\forall(N_1,N_2)\in\mathbb{N}^2.
$$
Taking $2N_1>d$ and $2N_2>m+d$ we obtain
$$
\overset{\circ}{a}_t(x,\eta)=\int_{\mathbb{R}^{2d}}e^{i<z,\zeta>}<z>^{-2N_1}(1-\Delta_\zeta)^{N_1}\left[<\zeta>^{-2N_2}(1-\Delta_z)^{N_2}\left(a(x+(1-t)z,x-tz,\eta+\zeta)\right)\right]dz\,\dbar\zeta.
$$
In order to estimate the derivatives of $\overset{\circ}{a}_t$ we just have to choose $N_1$ and $N_2$ as large as necessary.
\end{proof}

The following Lemma is the essential technical step in the proof of Theorem \ref{Main-1}.

\begin{lemma}\label{L-2_2}
 Let $b\in S^m(\mathbb{R}^{2d}\times\mathbb{R}^d)$ for some $m\in\mathbb{R}$ and let $f(x):=<x>^p$ for some fixed $p\in\mathbb{N}$. For any $\epsilon\in(0,1]$ we define $f_\epsilon(x):=f(\epsilon x)$. Then there exists a bounded family of symbols $\{s_\epsilon\}_{\epsilon\in(0,1]}$ in $S^{m-1}(\mathbb{R}^{2d}\times\mathbb{R}^d)$ that verify the following equality as linear operators on $\mathcal{S}(\mathbb{R}^d)$:
$$
f_\epsilon\mathfrak{E}(b)f_\epsilon^{-1}=\mathfrak{E}(b)+\epsilon\mathfrak{E}(s_\epsilon),\quad\forall\epsilon\in(0,1].
$$
\end{lemma}
\begin{proof}
 Let us start with the first order Taylor development:
$$
f_\epsilon(x)=f_\epsilon(y)+\left<x-y,\int_0^1\big(\nabla f_\epsilon\big)\big(y+t(x-y)\big)dt\right>
$$
and compute $f_\epsilon\mathfrak{E}(b)f_\epsilon^{-1}$ by using integrations by parts with respect to the variable $\eta$ and the equality
$$
(x-y)e^{i<x-y,\eta>}=D_\eta\big(e^{i<x-y,\eta>}\big)
$$
in order to obtain the fromula in the statement of Lemma \ref{L-2_2} with
$$
s_\epsilon(x,y,\eta):=i\epsilon^{-1}<x-y>^{-2N}\int_0^1\left<f_\epsilon(y)^{-1}\big(\nabla f_\epsilon\big)\big(y+t(x-y)\big),(1-\Delta)^N\big(\nabla_\eta b\big)(x,y,\eta)\right>dt,
$$
choosing $N$ such that $2N\geq p$. It is now easy to notice that for any $\alpha\in\mathbb{N}^d$ with $|\alpha|\geq1$ one can write
$$
\partial^\alpha\big(f_\epsilon^{\pm1}\big)=\epsilon f_\epsilon^{\pm1}g^\pm_{\epsilon,\alpha}
$$
with $g^\pm_{\epsilon,\alpha}\in BC^\infty(\mathbb{R}^d)$ uniformly with respect to $\epsilon\in(0,1]$. Moreover it is also evident due to the explicit form of $f(x)=<x>^p$ that there exists a positive constant $C$ such that
$$
f_\epsilon\big(y+t(x-y)\big)\leq Cf_\epsilon(y)<x-y>^p,\quad\forall(x,y)\in\mathbb{R}^d\times\mathbb{R}^d,\ \forall t\in[0,1],\ \forall\epsilon\in(0,1].
$$
It follows easily that $s_\epsilon\in S^{m-1}(\mathbb{R}^{2d}\times\mathbb{R}^d)$ uniformly with respect with $\epsilon\in(0,1]$.
\end{proof}
\begin{proof_} {\sf of Theorem \ref{Main-1}}

{\sf Point 1.} We shall apply Proposition \ref{P-1_1} to the operator $H$ associated to $\mathfrak{Op}^A(a)$. Under the assumptions of Theorem \ref{Main-1} we have that $D(H)=\mathcal{H}^m_A(\mathbb{R}^d)$ if $m>0$ and $D(H)=L^2(\mathbb{R}^d)$ if $m\leq0$.

Let us prove that the function $f_\epsilon(x)=<\epsilon x>^p$ from Lemma \ref{L-2_2} verifies the conditions 1-3 in Hypothesis \ref{Hyp-weight} as assumed in the hypothesis of Proposition \ref{P-1_1}. It is evident that $f_\epsilon^{-1}\in S^0(\mathbb{R}^d)$ so that the operator of multiplication with $f_\epsilon^{-1}$ in $L^2(\mathbb{R}^d)$ leaves invariant all the magnetic Sobolev spaces and thus satisfies Hypothesis \ref{Hyp-weight} (2). For the other two conditions in Hypothesis \ref{Hyp-weight} we shall use Lemma \ref{L-2_2} taking $b(x,y,\eta):=a\left(\frac{x+y}{2},\eta\right)$. In fact due to Lemma \ref{L-2_1} we have
$$
R_\epsilon:=\mathfrak{E}(s_\epsilon)=\mathfrak{Op}^A(r_\epsilon)
$$
with $\{r_\epsilon\}_{\epsilon\in(0,1]}$ a bounded family of symbols in $S^{m-1}(\mathbb{R}^d)$, so that
$$
R_\epsilon\in\mathbb{B}\left(\mathcal{H}^s_A(\mathbb{R}^d);\mathcal{H}^{s-m+1}_A(\mathbb{R}^d)\right),\quad\forall s\in\mathbb{R}
$$
uniformly with respect to $\epsilon\in(0,1]$. In particular, as operator in $L^2(\mathbb{R}^d)$ we have that $R_\epsilon\in\mathbb{B}\left(D(H);L^2(\mathbb{R}^d)\right)$ uniformly with respect to $\epsilon\in(0,1]$ and thus verifies Hypothesis  \ref{Hyp-weight} (1). Then Hypothesis \ref{Hyp-weight} (3) is a direct consequence of the equality in Lemma \ref{L-2_2}.

Using now Proposition \ref{P-1_1} we conclude that for any $p\in\mathbb{N}$ there exists $\epsilon_0\in(0,1]$ such that $<\epsilon x>^pu\in D(H^n)$ for any $n\geq1$ and any $\epsilon\in(0,\epsilon_0]$.

{\sf Point 2.} Suppose first that $m>0$. As for any $n\geq1$ we have $D(H^n)=\mathcal{H}^{nm}_A(\mathbb{R}^d)$ we conclude that for $u$ as in the statement of Theorem \ref{Main-1}
$$
<x>^pu\in\mathcal{H}^\infty_A(\mathbb{R}^d)=\underset{k\in\mathbb{N}}{\bigcap}\mathcal{H}^k_A(\mathbb{R}^d),\qquad\forall p\in\mathbb{N}.
$$
Taking now into account the description of the magnetic Sobolev spaces recalled in the Introduction and the fact that we could chosen $A$ with components of class $C^\infty_{\text{\sf pol}}(\mathbb{R}^d)$ we finaly conclude that $<x>^p\partial^\alpha u\in L^2(\mathbb{R}^d)$ for any $p\in\mathbb{N}$ and any $\alpha\in\mathbb{N}^d$ and thus $u\in\mathcal{S}(\mathbb{R}^d)$.

Let us consider now the case $m<0$ and $\lambda\ne0$. As a consequence of the proof of Point 1 we have seen that $r_\epsilon\in S^{m-1}(\mathbb{R}^d)$ and thus, due also to the formula in Lemma \ref{L-2_2} $H_\epsilon$ is a magnetic pseudodifferential operator of order $m$ for any $\epsilon\in(0,\epsilon_0]$. Going back to the proof of Proposition \ref{P-1_1} we can prove that for any $\epsilon\in(0,\epsilon_0]$ there exists $n_\epsilon\in\mathbb{N}^*$ such that for any $u\in\Ker(H-\lambda)$ we can find $v\in L^2(\mathbb{R}^d)$ such that $(H_\epsilon-\lambda)^{n_\epsilon}v=0$ and $u=f_\epsilon^{-1}v$. But if $\lambda\ne0$ the equation  $(H_\epsilon-\lambda)^{n_\epsilon}v=0$ implies that $v=Q_\epsilon v$ with $Q_\epsilon$ a magnetic pseudodifferential operator of order $m<0$ and thus $v\in\mathcal{H}^\infty_A(\mathbb{R}^d)$ and the proof can continue as in the case $m>0$.
\end{proof_}

\section{Proof of Theorem \ref{Main-2}}

The proof of the Theorem will be based on the following technical lemmas.

\begin{lemma}\label{L-3_1}
 Under the assumptions of Theorem \ref{Main-2} there exists a positive constant $C$ such that for any $\alpha\in\mathbb{N}^d$ the following inequality is true:
$$
\left|\big(\partial^\alpha_\eta a\big)(x,\eta)\right|\leq C\left(\frac{2}{\delta}\right)^{|\alpha|}\alpha!<\eta>^m,\quad\forall(x,\eta)\in\mathbb{R}^{2d}.
$$
\end{lemma}
\begin{proof}
 The above inequality follows directly from the Cauchy formula on a polydisc:
$$
\big(\partial^\alpha_\eta a\big)(x,\eta)=\frac{\alpha!}{(2\pi i)^d}\hspace{-5pt}\int\limits_{|\zeta_1-\eta_1|=\delta/2}\cdots\int\limits_{|\zeta_d-\eta_d|=\delta/2}\frac{\widetilde{a}(x,\zeta)}{(\zeta_1-\eta_1)^{1+\alpha_1}\cdots(\zeta_d-\eta_d)^{1+\alpha_d}}d\zeta_1\cdots d\zeta_d,\qquad\forall\eta=(\eta_1,\ldots,\eta_d)\in\mathbb{R}^d.
$$
\end{proof}

\begin{lemma}\label{L-3_2}
 Let us consider the function $f:\mathbb{R}^d\rightarrow[1,\infty)$ given by $f(x)=e^{<x>}$ and for any $\epsilon\in(0,1]$ let us define $f_\epsilon(x):=f(\epsilon x)$ for any $x\in\mathbb{R}^d$. Then the following equality is true
$$
f_\epsilon(x)f_\epsilon^{-1}(y)=\sum_{\alpha\in\mathbb{N}^d}\frac{\epsilon^{|\alpha|}b_\epsilon(x,y)^\alpha}{\alpha!}(x-y)^\alpha,\quad\forall(x,y)\in\mathbb{R}^d\times\mathbb{R}^d,\ \forall\epsilon\in(0,1],
$$
where $b_\epsilon:=(b_{\epsilon,1},\ldots,b_{\epsilon,d}):\mathbb{R}^{2d}\rightarrow\mathbb{R}^d$ has the properties:
\begin{enumerate}
 \item for any $j\in\{1,\ldots,d\}$ the component $b_{\epsilon,j}$ belongs to $BC^\infty(\mathbb{R}^d)$ uniformly with respect to $\epsilon\in(0,1]$,
\item for any $\epsilon\in(0,1]$ we have $|b_\epsilon(x,y)|\leq1$ for any $(x,y)\in\mathbb{R}^d\times\mathbb{R}^d$.
\end{enumerate}
\end{lemma}
\begin{proof}
 We just have to notice that
$$
f_\epsilon(x)f_\epsilon^{-1}(y)=\sum_{k=0}^\infty\frac{\big(<\epsilon x>-<\epsilon y>\big)^k}{k!}
$$
and that $<\epsilon x>-<\epsilon y>=\epsilon<x-y,b_\epsilon(x,y)>$ with 
$$
b_{\epsilon,j}(x,y):=\epsilon(x_j+y_j)\big(<\epsilon x>+<\epsilon y>\big)^{-1}
$$
evidently verifying the stated properties.
\end{proof}

\begin{lemma}\label{L-3_3}
 Suppose verified the hypothesis in Lemma \ref{L-3_1} and \ref{L-3_2}. We denote by $\epsilon_0:=\min\{1,\delta/4\}$ and for any $\epsilon\in(0,\epsilon_0]$ we define a function 
$$
c_\epsilon:\mathbb{R}^{2d}\times\mathbb{R}^d\rightarrow\mathbb{C},\quad c_\epsilon(x,y,\eta):=\widetilde{a}\left(\frac{x+y}{2},\eta+i\epsilon b_\epsilon(x,y)\right),\quad\forall(x,y,\eta)\in\mathbb{R}^{3d}.
$$
Then $c_\epsilon\in S^m(\mathbb{R}^{2d}\times\mathbb{R}^d)$ uniformly with respect to $\epsilon\in(0,\epsilon_0]$ and the following equality is true
$$
f_\epsilon\mathfrak{Op}^A(a)f_\epsilon^{-1}=\mathfrak{E}(c_\epsilon),\quad\epsilon\in(0,\epsilon_0].
$$
\end{lemma}
\begin{proof}
 The first statement follows by straightforward computations using the properties of $\widetilde{a}$ and $b_\epsilon$ (see Lemma \ref{L-3_2}). For the last equality in the statement let us choose $N\in\mathbb{N}$ such that $(N-1)d>m$, let us denote by
$$
I_N:=\left\{\alpha=(\alpha_1,\ldots,\alpha_d)\in\mathbb{N}^d\mid0\leq\alpha_j\leq N-1,\,1\leq j\leq d\right\}
$$
and let us rewrite the formula proved in Lemma \ref{L-3_2} as
$$
f_\epsilon(x)f_\epsilon^{-1}(y)=\sum_{\alpha\in I_N}\frac{\epsilon^{|\alpha|}b_\epsilon(x,y)^\alpha}{\alpha!}(x-y)^\alpha\,+\,(x_1-y_1)^N\cdots(x_d-y_d)^Ng_{\epsilon,N}(x,y).
$$
We notice that for any $\alpha\in\mathbb{N}^d\setminus I_N$ and for any $u\in C^\infty_0(\mathbb{R}^d)$ and any fixed $x\in\mathbb{R}^d$, the function
$$
\mathbb{R}^d\times\mathbb{R}^d\ni(y,\eta)\mapsto u(y)\big(\partial^\alpha_\eta a\big)\left(\frac{x+y}{2},\eta\right)\in\mathbb{C}
$$
is integrable. Thus, computing $f_\epsilon\mathfrak{Op}^A(a)f_\epsilon^{-1}$ by using the above formula for $f_\epsilon(x)f_\epsilon^{-1}(y)$ and integration by parts with respect to $\eta$ taking into account that
$$
(x-y)^\alpha e^{i<x-y,\eta>}=\big(-i\partial_\eta\big)^\alpha e^{i<x-y,\eta>},
$$
one obtains for any $u\in C^\infty_0(\mathbb{R}^d)$
$$
f_\epsilon\mathfrak{Op}^A(a)f_\epsilon^{-1}u=S_\epsilon u+T_\epsilon u
$$
with
\begin{equation}\label{S}
 \big(S_\epsilon u\big)(x):=\int_{\mathbb{R}^{2d}}e^{i<x-y,\eta>}\omega^A(x,y)\left[\sum_{\alpha\in I_N}\frac{\epsilon^{|\alpha|}b_\epsilon(x,y)^\alpha}{\alpha!}\big(i\partial_\eta\big)^\alpha a\left(\frac{x+y}{2},\eta\right)\right]u(y)\,dy\,\dbar\eta
\end{equation}
\begin{equation}\label{T}
 \big(T_\epsilon u\big)(x):=\int_{\mathbb{R}^d}\omega^A(x,y)u(y)\sum_{\alpha\in \mathbb{N}^d\setminus I_N}\frac{\epsilon^{|\alpha|}b_\epsilon(x,y)^\alpha}{\alpha!}\left[\int_{\mathbb{R}^d}e^{i<x-y,\eta>}\big(i\partial_\eta\big)^\alpha a\left(\frac{x+y}{2},\eta\right)\dbar\eta\right]dy.
\end{equation}
We apply Lemma \ref{L-3_1} for $\partial_{\eta_1}^N\cdots\partial_{\eta_d}^Na\in S^{m-Nd}(\mathbb{R}^d)$ and conclude the following inequality:
$$
\sum_{\alpha\in \mathbb{N}^d\setminus I_N}\frac{\epsilon^{|\alpha|}}{\alpha!}\left|b_\epsilon(x,y)^\alpha\big(i\partial_\eta\big)^\alpha a\left(\frac{x+y}{2},\eta\right)\right|\leq C_N<\eta>^{m-Nd},\quad\forall(x,y,\eta)\in\mathbb{R}^{3d},\,\forall\epsilon\in(0,\epsilon_0],
$$
where $C_N$ is a positive constant. This inequality allows to permute in \eqref{T} the sum with respect to $\alpha$ with the integral with respect to $\eta$ and thus the equality in the conclusion of Lemma \ref{L-3_3} follows after a Taylor development of the map 
$$
t\mapsto\widetilde{a}\left(\frac{x+y}{2},\eta+it b_\epsilon(x,y)\right)
$$
by using the above formulae \eqref{S} and \eqref{T}.
\end{proof}
\begin{lemma}\label{L-3_4}
 Under the assumptions of Lemma \ref{L-3_3} there exists a bounded family of symbols $\left\{r_\epsilon\right\}_{\epsilon\in(0,\epsilon_0]}\subset S^{m-1}(\mathbb{R}^d)$ such that
$$
f_\epsilon\mathfrak{Op}^A(a)f_\epsilon^{-1}u\,=\,\mathfrak{Op}^A(a)\,+\,\epsilon\mathfrak{Op}^A(r_\epsilon),\quad\forall\epsilon\in(0,\epsilon_0].
$$
\end{lemma}
\begin{proof}
 The family of symbols $\{c_\epsilon\}_{\epsilon\in(0,\epsilon_0]}$ appearing in Lemma \ref{L-3_3} may be rewritten as
$$
c_\epsilon(x,y,\eta)\,=\,a\left(\frac{x+y}{2},\eta\right)\,+\,\epsilon d_\epsilon(x,y,\eta),\quad\forall(x,y,\eta)\in\mathbb{R}^{3d},\,\forall\epsilon\in(0,\epsilon_0]
$$
with
$$
d_\epsilon(x,y,\eta):=i\int_0^1\left<b_\epsilon(x,y),\big(\nabla_\eta\widetilde{a}\big)\left(\frac{x+y}{2},\eta+it\epsilon b_\epsilon(x,y)\right)\right>dt.
$$
It is evident that $d_\epsilon\in S^{m-1}(\mathbb{R}^{2d}\times\mathbb{R}^d)$ uniformly with respect to $\epsilon\in(0,\epsilon_0]$. Thus, due to Lemma \ref{L-2_1} there exists a bounded family $\left\{r_\epsilon\right\}_{\epsilon\in(0,\epsilon_0]}\subset S^{m-1}(\mathbb{R}^d)$ such that $\mathfrak{E}(d_\epsilon)=\mathfrak{Op}^A(r_\epsilon)$. The equality in the conclusion of the Lemma now clearly follows from the one in Lemma \ref{L-3_3} and the above formula of $c_\epsilon$.
\end{proof}

\begin{proof_} {\sf of Theorem \ref{Main-2}.}
 We proceed exactly as in the proof of Theorem \ref{Main-1} using Proposition \ref{P-1_1} and Lemma \ref{L-3_4}.
\end{proof_}

\section{Proof of Theorem \ref{Main-3}}

{\sf Point 1.} It is sufficient to prove that there exists $\epsilon_0\in(0,1]$ such that for the operator $H=H_A\dot{+}V$, the functions $f_\epsilon(x):=e^{<\epsilon x>}$ with any $\epsilon\in(0,\epsilon_0]$, verify the conditions (1)-(3) in Hypothesis \ref{Hyp-weight}.

The operator $H_A$ verifies the hypothesis of Theorem \ref{Main-2} for $a(x,\xi):=<\xi>$ and $m=1$. Due to Lemma \ref{L-3_4} there exists $\epsilon_0\in(0,1]$ and a bounded family of operators $\{R_\epsilon\}_{\epsilon\in(0,\epsilon_0]}$ in $\mathbb{B}\big(L^2(\mathbb{R}^d)\big)$ such that for any $u\in D(H_A)=\mathcal{H}^1_A(\mathbb{R}^d)$ one has $f_\epsilon^{-1}u\in D(H_A)$ and
$$
f_\epsilon H_Af_\epsilon^{-1}u=H_Au+\epsilon R_\epsilon u,\quad\forall\epsilon\in(0,\epsilon_0].
$$
Thus assumtion (1) in Hypothesis \ref{Hyp-weight} is trivialy satsified and we have to prove assumptions (2) and (3).

Let us choose $u\in D(H)$ and $\epsilon\in(0,\epsilon_0]$. It follows that $u\in D(h)$ and thus $u\in\mathcal{H}^{1/2}_A(\mathbb{R}^d)$ and $V_+u\in L^2(\mathbb{R}^d)$. Then we also have $V_+f_\epsilon^{-1}u\in L^2(\mathbb{R}^d)$ and due to the fact that $f_\epsilon^{-1}\in S^0(\mathbb{R}^d)$ we also obtain $f_\epsilon^{-1}u\in\mathcal{H}^{1/2}_A(\mathbb{R}^d)$. We conclude that $f_\epsilon^{-1}u\in D(h)$. The above equality then implies
$$
\mathfrak{Op}^A(a)\big(f_\epsilon^{-1}u\big)+V\big(f_\epsilon^{-1}u\big)=f_\epsilon^{-1}\left[\mathfrak{Op}^A(a)u+Vu+\epsilon R_\epsilon u\right]\in L^2(\mathbb{R}^d),\quad\forall\epsilon\in(0,\epsilon_0].
$$
We conclude that $f_\epsilon^{-1}u\in D(H)$ and the above equality just implies assumption (3) in Hypothesis \ref{Hyp-weight}.

{\sf Point 2.} For $V_-\in\mathcal{K}_d$ the operator $H\equiv H(A,V)$ has the following two properties proved in \cite{IMP2}:
\begin{enumerate}
 \item For any $u\in L^2(\mathbb{R}^d)$ and for any $t\geq0$ the following inequality is true:
$$
\left|e^{-tH}u\right|\leq e^{-tH(0,-V_-)}|u|,\ \text {a.e. on }\mathbb{R}^d.
$$
\item The operator $e^{-H(0,-V_-)}$ has an integral kernel that verifies the estimations:
$$
\forall p>1,\ \exists C_p>0\ \text{such that }0\leq e^{-H(0,-V_-)}(x,y)\leq C_pe^{-<x-y>/p},\ \forall(x,y)\in\mathbb{R}^{2d}.
$$
We conclude then that for $\epsilon_0\in(0,1]$ from point (1), for $u$ and $\lambda$ as in the statement of the Theorem, for $p\in(1,\epsilon_0^{-1})$, $\epsilon\in(0,\epsilon_0]$ and $x\in\mathbb{R}^d$ we have
$$
\left|f_\epsilon(x)u(x)\right|=e^\lambda f_\epsilon(x)\left|\big(e^{-H}u\big)(x)\right|\leq e^\lambda f_\epsilon(x)\left(e^{-H(0,-V_-)}|u|\right)(x)\leq
$$
$$
\leq C_pe^\lambda\int_{\mathbb{R}^d}e^{-|x-y|(p^{-1}-\epsilon)}f_\epsilon(y)|u(y)|dy\leq C_pe^\lambda\left(\int_{\mathbb{R}^d}e^{-2|z|(p^{-1}-\epsilon_0)}dz\right)^{1/2}\|f_\epsilon u\|_{L^2(\mathbb{R}^d)}.
$$
\end{enumerate}
\hfill $\blacksquare $

\noindent{\bf Acknowledgements: } RP acknowledges partial support from the ANCS Partnership Contract no. 62-056/2008.

E-mail:Viorel.Iftimie@imar.ro, Radu.Purice@imar.ro

\end{document}